%% file: main.tex
\documentclass[10pt]{article}

\usepackage{amsmath,amssymb,amsthm}
\usepackage{mathrsfs}
\usepackage{authblk}

\usepackage[left=2cm,right=2cm,top=2cm,bottom=2cm,bindingoffset=0cm]{geometry}
\usepackage{nccfoots}
\usepackage[hang,flushmargin]{footmisc} 
\usepackage{hyperref}

\usepackage{graphicx}

\usepackage{pgf,tikz}
\usetikzlibrary{decorations.pathmorphing}
\usepackage{xcolor}

\newtheorem{example}{Example}

\newtheorem{proposition}{Proposition}
\newtheorem{lemma}{Lemma}
\newtheorem{corollary}{Corollary}

\newtheorem{theorem}{Theorem}
\newtheorem{problem}{Problem}
\newtheorem{definition}{Definition}

\DeclareMathOperator{\supp}{supp}

\title{On branching points in the Gilbert--Steiner problem}

\author{Danila Cherkashin}

\affil{Institute of Mathematics and Informatics, Bulgarian Academy of Sciences, Sofia}

\begin{document}

\maketitle

\begin{abstract}
The Gilbert--Steiner problem is a generalization of the Steiner tree problem and specific optimal mass transportation, which allows the use additional (branching) point in a transport plan.
A specific feature of the problem is that the cost of transporting a mass $m$ along a segment of length $l$ is equal to $l \times m^p$ for a fixed $0 < p < 1$ and 
segments may end at points not belonging to the supports of given measures (branching points).

Main result of this paper determines all pairs of $(p,d)$ for which the Gilbert--Steiner problem in $\mathbb{R}^d$ admits only branching points of degree 3. 
Namely, it happens if and only if $d = 2$ or $p < 1/2$. 
\end{abstract}

\section{Introduction}


One of the first models for branched transport was introduced by Gilbert~\cite{gilbert1967minimum}. The difference from the optimal transportation problem is that the extra geometric points may be of use; this explains the naming in honor of Steiner. Sometimes it is also referred to as \textit{optimal branched transport}; a large part of book~\cite{bernot2008optimal} is devoted to this problem. Let us proceed with the formal definition.

\begin{definition}
    Let $\mu^+,\mu^-$ be two finite measures on a metric space $(X,\rho(\cdot,\cdot))$ with finite supports such that total masses $\mu^+(X)=\mu^-(X)$ are equal. Let $V\subset X$  be a finite set containing the support of the signed measure $\mu^+-\mu^-$, the elements of $V$ are called \emph{vertices}. Further, let  $E$ be a finite collection of unordered pairs $\{x,y\}\subset V$ which we call \emph{edges}. So, $(V,E)$ is a  simple undirected finite graph. Assume that for every $\{x,y\}\in E$ two non-zero real numbers $m(x,y)$ and $m(y,x)$ are defined so that $m(x,y)+m(y,x)=0$. This data set is called a \emph{$(\mu^+,\mu^-)$-flow} if 
    \[
    \mu^+ - \mu^- = \sum_{\{x,y\}\in E} m(x,y)\cdot (\delta_y-\delta_x)
    \]
    where $\delta_x$ denotes a delta-measure at $x$ (note that the summand $m(x,y)\cdot (\delta_y-\delta_x)$ is well-defined in the sense that it does not depend on the order of $x$ and $y$).
\end{definition}

Let $c\colon [0,\infty)\to [0,\infty)$ be a \emph{cost function}. The expression
\[
\sum_{\{x,y\}\in E} c(|m(x,y)|) \cdot \rho(x,y)
\]
is called the \emph{Gilbert functional} of the $(\mu^+,\mu^-)$-flow.

\begin{problem}[General finite Gilbert--Steiner problem] \label{pr:gen}
Find a $(\mu^+,\mu^-)$-flow which minimizes the Gilbert functional for $\mu^+,\mu^-$ with finite supports and a monotone, concave and non-negative function $c(x)$ such that $c(0)=0$.
\end{problem}

A solution to problem~\ref{pr:gen} is called a \emph{minimal flow}.
Note that concavity and non-negativity imply subadditivity, that is $c(m_1 + m_2) \leq c(m_1) + c(m_2)$.
The greatest interest is attracted by the case of the cost function $x^p$.

\begin{problem}[Gilbert--Steiner problem]
Find a $(\mu^+,\mu^-)$-flow which minimizes the Gilbert functional for $\mu^+,\mu^-$ with finite supports and $c(x) = x^p$ for $0 < p < 1$.
\end{problem}

Vertices from $\supp(\mu^+) \cup \supp(\mu^-)$ are called \textit{terminals}.
A vertex from $V \setminus \supp(\mu^+) \setminus \supp(\mu^-)$ is called a \textit{branching point}. Formally, we allow a branching point to have degree 2, but such points may be easily eliminated.

Local structure in the Gilbert--Steiner problem was discussed in~\cite{bernot2008optimal}, and the paper~\cite{lippmann2022theory} deals with planar case. A local picture around a branching point $b$ of degree 3 is clear due to the initial paper of Gilbert.
Similarly to the finding of the Fermat--Torricelli point in the celebrated Steiner problem one can determine the angles around $b$ in terms of masses (see Lemma~\ref{main_l}).

\begin{theorem}[Cherkashin--Petrov, 2025~\cite{cherkashin2025branching}]
\label{theorem:main}
Suppose that the cost function $c(\cdot)$ satisfies the decomposition
\begin{equation}\label{def_of_adm_fun}
    c(t)=\sqrt{\int \sin^2 (tx)\, d\lambda(x)}=\frac12\|e^{2itx}-1\|_{L^2(\lambda)},
\end{equation}
where $\lambda$ be a Borel measure on $\mathbb{R}$ with uncountable support such that $\int \min(x^2,1)d\lambda(x)<\infty$.

Then a solution to the planar general Gilbert--Steiner problem for $c(x)$ has no branching point of degree at least 4. In particular, this holds for $c(x) = x^p$ (via $d\lambda = const \cdot x^{-2p-1}dx$).
\end{theorem}

Note that condition~\eqref{def_of_adm_fun} implies $c(0)=0$, non-negativity and subadditivity of 
$c$; on the other hand it does not imply monotonicity.

It is known that an optimal flow for a strictly concave $c(\cdot)$ does not contain cycles (see Proposition 7.8 in~\cite{bernot2008optimal}). Thus it can be represented as a collection of trees.

An \textit{irrigation} (or arborescence) setup, which restricts $\mu^+$ to be a Dirac measure was independently considered in Minkowski space by Volz, Brazil, Ras, Swanepoel and Thomas~\cite{volz2013gilbert}. The arborescence naming is given because in this case an optimal flow corresponds to a tree.
In a Euclidean space (of an arbitrarily dimension) they obtained that a degree of a branching is at most 3 provided that $c(\cdot)$ is 
a concave monotone function, the function $(c^2)'$ is convex and $c(0) > 0$. Note that for $p \in (1/2,1)$ the convexity condition fails.

Main result of this paper is the following theorem.

\begin{theorem} \label{th:main}
If $d = 2$ or $p < 1/2$ then degree of any branching point in a solution to the Gilbert--Steiner problem in $\mathbb{R}^d$ 
is equal to three.
If $p \geq 1/2$ then there is a solution to the Gilbert--Steiner problem consisting of $d+1$ segment adjacent to the same branching point.
\end{theorem}

\paragraph{Structure of the paper.} Section~\ref{sec:pre} contains well-known lemmas which are used in our proofs.
Section~\ref{sec:upper} adopt and simplify methods from~\cite{volz2013gilbert}; namely we show (Theorems~\ref{th:stupid1} and~\ref{th:stupid2}) that the condition $c(0) > 0$ can be replaced with strict convexity of $(c^2)'$ which is more natural in the context of optimal transportation and that the result can be generalized from irrigation to the general case.
Section~\ref{sec:lower} provides an example of branching with degree $d + 1$ for $p \geq 1/2$ in $d$-dimensional Euclidean space. In Section~\ref{sec:source} we study behavior of a solution in a neighborhood of a terminal point and finally Section~\ref{sec:discussion} discusses applications and open problems.

\section{Preliminaries} \label{sec:pre}

We will need the following lemmas. 

\begin{lemma}[Folklore]  \label{main_l}
Let $PQR$ be a triangle and $w_1$, $w_2$, $w_3$ be non-negative reals. 
For every point $X \in \mathbb{R}^2$ consider the value
\[
L(X) := w_1 \cdot  |PX| + w_2 \cdot  |QX| + w_3 \cdot  |RX|.
\]
Then 
\begin{itemize}
    \item [(i)] a minimum of $L(X)$ is achieved at a unique point $X_{min}$;
    \item [(ii)] if $X_{min} = P$ then $w_1 \geq w_2 + w_3$ or there is a triangle $\Delta$ with sides $w_1$, $w_2$, $w_3$ and $\angle P$ is at least the exterior angle between $w_2$ and $w_3$ in $\Delta$.
\end{itemize}
\end{lemma}

Consider a Gilbert flow and two adjacent edges with masses $m_1$ and $m_2$ (they may be of different signs).
Lemma~\ref{main_l} compares it with a blue tripod competitor with additional branching point, see Fig.~\ref{fig:maincomp}.

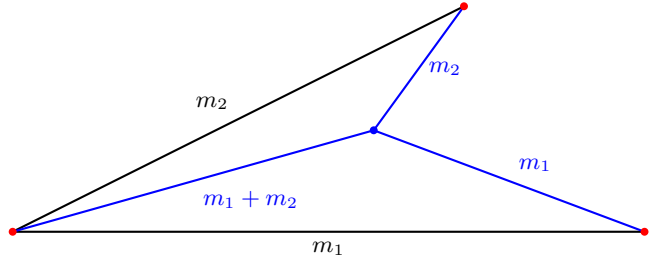
\begin{figure}[h]
    \centering  
    \input{pictures/mainCompet.tex}
    \caption{Basic optimality test: comparison of two adjacent edges and a tripod}
    \label{fig:maincomp}
\end{figure}

Another application of Lemma~\ref{main_l} immediately describes the behavior of a minimal flow in a neighborhood of a triple branching. The angles around the point are equal to the exterior angles of a triangle with sides $c(m_1)$, $c(m_2)$ and $c(m_3)$ (where outsource $m_1+m_2+m_3 = 0$).

\begin{definition}
We say that a non-increasing $n$-tuple $x_1, \dots, x_n$ \textit{majorizes} a non-increasing $n$-tuple $y_1, \dots, y_n$ if
the inequalities
\[
x_1 + \dots + x_i \geq y_1 + \dots + y_i \quad \text{hold for all } i \in \{1, \dots, n-1\},
\]
and 
\[
x_1 + \dots + x_n = y_1 + \dots + y_n.
\]
\end{definition}

\begin{lemma}[Karamata's inequality] \label{lm:Karineq}
Let $I \subset \mathbb{R}$ be an interval and let $f : I \to \mathbb{R}$ be a convex function. If $x_1, \dots, x_n, y_1, \dots, y_n \in I$ are such that 
$(x_1, \dots, x_n)$ majorizes $ (y_1, \dots, y_n)$, then
\[
f(x_1) + \cdots + f(x_n) \geq f(y_1) + \cdots + f(y_n).
\]

If $f$ is strictly convex, then the equality holds if and only if we have $x_i = y_i$ for all $i \in \{1, \dots, n\}$.
\end{lemma}

\section{Upper bounds on the degree of a branching point} \label{sec:upper}

Hereafter we consider a Euclidean space $\mathbb{R}^d$, $d \geq 2$. Recall that a cost function $c(x)$ is always concave, increasing and $c(0) = 0$.

\begin{definition}
 We say that the set of unit vectors $e_1, \dots e_k \in \mathbb{R}^d$ with masses $m_1,\dots,m_k$ is \emph{satisfactory configuration} if $m_1 + \dots + m_k = 0$ and the inequality
 \begin{equation} \label{eq:sat}
 \langle e_i,e_j \rangle \leq \frac{c^2(m_i + m_j) - c^2(m_i) - c^2(m_j)}{2c(m_i)c(m_j)} 
 \end{equation}
 holds for every $i \neq j$.
\end{definition}

Obviously, if $A$ is a source of a branching point of an optimal flow then the unit vectors emanating along the flow edges adjacent to $A$ with the corresponding masses produces a satisfactory configuration. Otherwise, by Lemma~\ref{main_l} a tripod is shorter than a pair of edges which fails~\eqref{eq:sat}. Thus a local transformation produces a better competitor, contradiction.

\begin{theorem} \label{th:stupid1}
Let $c$ be a cost function such that $c^2(x)$ has a strictly convex derivative. 
Then there is no satisfactory configuration with unique negative mass and $k > 3$. 
If $k=3$ then $e_1,e_2$ and $e_3$ are coplanar.
\end{theorem}

\begin{proof}
Consider a satisfactory configuration with unique negative mass, say, $m_1$. Clearly
\[
0 \leq \left \langle \sum_{i=1}^k c(m_i)e_i, \sum_{i=1}^k c(m_i)e_i \right \rangle = \sum_{i=1}^k c^2(m_i) + 2\sum_{1 \leq i < j \leq k} c(m_i)c(m_j) \langle e_i,e_j \rangle \leq
\]
\[
\sum_{i=1}^k c^2(m_i) + \sum_{1 \leq i < j \leq k} \left( c^2(m_i + m_j) - c^2(m_i) - c^2(m_j) \right)
\]
by the property of satisfaction.
Substitute $m_1 = - m_2 - \dots - m_k$ and use $c(-x) = c(x)$ to get
\[
0 \leq c^2 \left (\sum_{i=2}^k m_i \right) + \sum_{i=2}^k c^2(m_i) + \sum_{2 \leq i \leq k}  \left ( c^2 \left ( \left ( \sum_{j=2}^k m_j \right ) - m_i \right ) - c^2(m_i) - c^2 \left (\sum_{j=2}^k m_j \right ) \right ) +
\]
\[
\sum_{2 \leq i < j \leq k} \left( c^2(m_i + m_j) - c^2(m_i) - c^2(m_j) \right) =: f (m_2, \dots, m_k) = 
\]
\[
 \sum_{2 \leq i \leq k}    c^2 \left ( \left ( \sum_{j=2}^k m_j \right ) - m_i \right ) + \sum_{2 \leq i < j \leq k}  c^2(m_i + m_j) - (k-2) \sum_{i=2}^k c^2(m_i) - (k-2) c^2 \left (\sum_{i=2}^k m_i \right).
\]
Let $g(x)$ denote $(c^2(x))'$. Then
\[
\frac{df}{dm_2} = \sum_{3 \leq i \leq k} g \left ( \left ( \sum_{j=2}^k m_j \right ) - m_i \right ) +  \sum_{3 \leq i \leq k} g(m_2 + m_i) - (k-2) g(m_2) - (k-2) g \left (\sum_{i=2}^k m_i \right).
\]

Note that the $(2k-4)$-tuple $x_1, \dots, x_{k-2} = \sum_{i=2}^k m_i$, $x_{k-3}, \dots, x_{2k-4} = m_2$ majorizes any $(2k-4)$-tuple with entries in the interval $[m_2, \sum_{i=2}^k m_i]$, in particular the ordered non-increasing $(2k-4)$-tuple with entries $\left ( \sum_{j=2}^k m_j \right ) - m_i, m_2 + m_i$, where $3 \leq i \leq k$.
Since $g$ is strictly convex on $[0,+\infty)$ by Lemma~\ref{lm:Karineq} we have $\frac{df}{dm_2} < 0$. Thus we have $0 \leq f(0,\dots,0) = 0$, so we have equalities in all inequalities including the application of Lemma~\ref{lm:Karineq}. In particular the second $(2k-4)$-tuple contains $\sum_{i=2}^k m_i$ which is possible only for some $m_2 + m_i$, which means $k \leq 3$.

In the case $k=3$ the equalities 
\[
\langle e_i,e_j \rangle = \frac{c^2(m_i + m_j) - c^2(m_i) - c^2(m_j)}{2c(m_i)c(m_j)} 
\]
imply that angles between $e_1$, $e_2$ and $e_3$ are equal to exterior angle of the triangle with sides $c(m_1)$, $c(m_2)$ and $c(m_3)$, so there sum is $2\pi$ and they are coplanar.
\end{proof}

\begin{theorem} \label{th:stupid2}
Let $c(x)$ be a differentiable function such that $(c^2)'$ is strictly convex. Then there is no satisfactory configuration with $k > 3$. 
\end{theorem}

\begin{proof}
Suppose that there is a satisfactory configuration with vectors $e_1, \dots, e_{k+l}$, positive masses $m_1,\dots, m_k$ and negative masses $-n_1, \dots -n_l$. Without loss of generality let
\[
m_1 = \max(m_1,\dots,m_k, |n_1|, \dots |n_l|).
\]
Erase the vectors $e_2, \dots, e_k$ with the corresponding masses and replace $m_1$ with $M := m_1 + \dots + m_k$.
We claim that vectors $e_1, e_{k+1}, \dots, e_{k+l}$ and masses $M, -n_1, \dots, -n_{l}$ form a satisfactory configuration.

Indeed, the bounds~\eqref{eq:sat} on $\langle e_i, e_j \rangle$ for $i,j > k$ do not change so let us analyze the bounds~\eqref{eq:sat} on $\langle e_1,e_i \rangle$ for $i > k$.
Take the derivative
\[
\frac{d}{dm} \frac{c^2(m-n) - c^2(m) - c^2(n)}{2c(m)c(n)} = \frac{2c'(m-n)c(m-n)c(m) - c'(m)(c^2(m) + c^2(m-n)-c^2(n))}{2c^2(m)c(n)}.
\]
Since $c(x)$ is concave and monotone $c'(m-n) > c'(m) > 0$ and $2c(m-n)c(m) > c^2(m) + c^2(m-n) - c^2(n)$ because
\[
\frac {c^2(m) + c^2(m-n) - c^2(n)}{2c(m-n)c(m)}
\]
is a cosine of (acute) angle in the triangle with sides $c(m), c(m-n), c(n)$ which exists due to subadditivity of $c(x)$. So the derivative is positive and thus every condition on the angle between $e_1$ and $e_i$, $i > k$ becomes milder.

Theorem~\ref{th:stupid1} says that $l = 2$ (if $l = 1$, then $n_1$ has maximal absolute value) and if $e_1$, $e_{k+1}, e_{k+2}$ are coplanar with equalities in the inequality of satisfaction. If $k > 2$ than $M > m_1$ and these equalities break.

\end{proof}

\begin{corollary}
The Gilbert--Steiner problem for an arbitrary $d$ and $p < 1/2$ admits only triple branching.
\end{corollary}

\begin{proof}
Recall that by Lemma~\ref{main_l} a branching point of degree $k$ produces a satisfactory configuration which consists of $k$ vectors. 
By Theorem~\ref{th:stupid2} this is impossible for $p < 1/2$.
\end{proof}

\begin{proposition}
For $p = 1/2$ the maximal degree of a branching point in a $d$-dimensional solution to the Gilbert--Steiner problem is at most $d+1$.
\end{proposition}

\begin{proof}
Let the vectors parallel to edges emanating from the branching point be $e_1, \dots, e_k$.
The satisfactory inequality~\eqref{eq:sat} in the case $p = 1/2$ says that the angles between masses of the same sign are at least $\pi/2$ and the angles between masses of different signs are strictly greater than $\pi/2$.

We use induction in $d$. The base $d = 2$, $k \leq 3$ is obvious, since at least one angle should be obtuse. For the step choose any vector $v \in \{e_1,\dots,e_k\}$ and project everything onto the hyperplane $v^\perp$. Note that the scalar product of projections of any vectors from $\{e_1,\dots,e_k\} \setminus v$ decreases. 
Also no two vectors can share the projection. If zero is not represented as a projection we are immediately done. 
In the remaining case if $v, -v$ have masses of the same sign then there is no vector with mass of the opposite sign; otherwise there are no more vectors at all.
\end{proof}

\section{Examples} \label{sec:lower}

The following theorem was proved in~\cite{volz2013gilbert} for a general Minkowski space (the authors do not use the condition $c(0) > 0$ through the proof).

\begin{theorem} \label{th:volz}
Consider a single source at $q$ and sinks $p_1,\dots,p_n \in \mathbb{R}^d$. 
Let the mass demand at $p_i$ be $m_i$ for $i=1,\dots,n$. 
Then a star flow with a unique branching point at the origin $o$ is optimal if and only if
\begin{equation}\label{eq:wstp}
\sum_{i=1}^n c(m_i) e_i + c\left( \sum_{i=1}^n m_i \right) e = \bar{0}
\end{equation}
and
\begin{equation}\label{eq:split1}
\left \langle \sum_{i\in I} c(m_i)e_i,\sum_{i\in I} c(m_i)e_i \right \rangle \leq c^2 \left (\sum_{i\in I} m_i \right ) \quad \text{ for all } I\subseteq\{1,\dots,n\},
\end{equation}
where $e_i$ is the unit vector codirected with $op_i$ and $e$ is the unit vector codirected with $oq$.
\end{theorem}

Note that Theorem~\ref{th:volz} means that local optimality of a star implies global optimality.

\begin{example} \label{ex:dplus1}
Let $e_1, \dots, e_d$ be a standard basis of $\mathbb{R}^d$ and $c(x) = \sqrt{x}$. 
Consider any positive masses $m_1,\dots,m_d$ and define $e$ by 
\[
\sum_{i=1}^n \sqrt{m_i} \cdot  e_i + \sqrt{  \sum_{i=1}^n m_i } \cdot e = \bar{0}.
\]
Clearly, $e$ has unit length, so~\eqref{eq:wstp} holds. Also for every $I\subseteq\{1,\dots,d\}$ one has an equality in~\eqref{eq:split1}.

Thus for $p = 1/2$ there is an example of $(d+1)$-branching.
\end{example}

This example can be extended to the case $p \geq 1/2$ as follows:

\begin{example} \label{ex:dplus1p}
Let $e_1, \dots, e_d$ be unit vectors in $\mathbb{R}^d$ and $c(x) = x^p$ such that $\langle e_i, e_j \rangle = a$,
where
\[
\sqrt{d(1 + (d-1)a)} = d^p \quad \quad \mbox{so that} \quad \quad a = \frac{d^{2p-1}-1}{d-1}.
\]
Consider masses $m_1 = \dots = m_d = 1$ and define $e$ by 
\[
\sum_{i=1}^n  e_i + d^p \cdot e = \bar{0}.
\]
Clearly, $e$ has unit length, so~\eqref{eq:wstp} holds. 
Inequality~\eqref{eq:split1} is equivalent to
\[
|I| + |I|\cdot(|I| - 1) a \leq |I|^{2p}.
\]
and
\[
\frac{d^{2p-1}-1}{d-1} \leq \frac{|I|^{2p-1}-1}{|I|-1}
\]
The derivative 
\[
\frac{d}{dx} \frac{x^{2p-1}-1}{x-1} = \frac{1 - (2p-1) x^{2p-2} + (2p-2) x^{2p-1}}{(x - 1)^2}.
\]
The denominator is always positive, so
\[
\frac{d}{dx} \left ( 1 - (2p-1) x^{2p-2} + (2p-2) x^{2p-1} \right ) = (2p-2)(2p-1) x^{2p-3} (x-1) < 0
\]
means that $1 - (2p-1) x^{2p-2} + (2p-2) x^{2p-1} < 1 - (2p-1) + (2p-2) = 0$, so $\frac{x^{2p-1}-1}{x-1}$ decreases. 

Thus for $p \geq 1/2$ there is an example of $(d+1)$-branching.
\end{example}

\section{Degree of a source} \label{sec:source}

The following theorem is a version of Theorem~\ref{th:volz} which bounds the degree of a source.

\begin{theorem}\label{th:stupid3}
Suppose that sinks $p_1,\dots,p_n\in\mathbb{R}^d$, and a single source in the origin $o$. Let the flow demand at $p_i$ be $m_i$ for $i=1,\dots,n$. 
Then a star flow is optimal if and only if
\begin{equation}\label{eq:split}
\left \langle \sum_{i\in I} c(m_i)e_i,\sum_{i\in I} c(m_i)e_i \right \rangle \leq c^2 \left (\sum_{i\in I} m_i \right ) \quad \text{ for all } I\subseteq\{1,\dots,n\},
\end{equation}
where $e_i$ is a unit vector codirected with $op_i$.
\end{theorem}

\begin{proof}
($\Rightarrow$) 
Assume without loss of generality that $I\neq\varnothing$. Consider a competitor in which some mass goes from the source point $o$ via vector $te$, for some $t\in\mathbb{R}_+$ and a unit vector $e$. In this new network, for each index $i\notin I$, the source $p_i$ remains connected to $o$ with flow demand $m_i$ and for indices $i\in I$, the source $p_i$ is now connected to $te$ with flow demand $m_i$, and $te$ is connected to $o$ with flow demand $\sum_{i\in I} m_i$. Minimality of the original network then implies that, for all unit vectors $e$, the function
\[
\psi_e(t) = \sum_{i\in I} c(m_i) \cdot \left(\|p_i - te\| - \|p_i\| \right) + c\left (\sum_{i\in I} m_i\right) |t|
\]
attains its minimum at $t=0$. Although $\psi_e$ may not be differentiable at $0$, we can still take the right-hand derivative to conclude that
\[
0 \leq\lim_{t\to0^+} \frac{\psi_e(t)}{t} 
= \lim_{t\to0^+} \left( \sum_{i\in I} c(m_i) \frac{\|p_i-te\|-\|p_i\|}{t} + c\left (\sum_{i\in I} m_i\right) \right).
\]
This inequality yields
\[
\left\langle \sum_{i\in I} c(m_i)e_i, e\right\rangle \leq c \left (\sum_{i\in I} m_i\right).
\]
Substituting  $e$ to be parallel $\sum_{i\in I} c(m_i)e_i$, be obtain the desired inequality.

($\Leftarrow$) Now assume that a flow satisfies~\eqref{eq:split}. Let $T$ be any competitor for the given data. 
For each $i=1,\dots,n$, let $P_i$ denote the unique path in $T$ from $p_i$ to $o$, say
\[
P_i = x^{(i)}_1 x^{(i)}_2 \dots x^{(i)}_{k_i}, \quad x^{(i)}_1 = p_i,\quad x^{(i)}_{k_i}=o,
\]
with the edges $x^{(i)}_j x^{(i)}_{j+1}$ distinct for $j=1,\dots,k_i-1$. For each edge $e$ of $T$, define the index set 
$S_e = \{ i : e \text{ lies on } P_i\}$.
Then the flow on edge $e$ is $\sum_{i\in S_e} m_i$, and its cost is
\[
c\left (\sum_{i\in S_e} m_i\right)\,\|x-y\|,
\]
where $e$ connects vertices $x$ and $y$. Hence, the total cost of $T$ is
\[
c(T) := \sum_{e=xy \text{ edge of }T} c\left (\sum_{i\in S_e} m_i\right)\,\|x-y\|. 
\]

On the other hand, the cost of the star network is
\[
\sum_{i=1}^n c(m_i) \langle e_i, p_i \rangle = \sum_{i=1}^n c(m_i) \sum_{j=1}^{k_i-1} \langle e_i, x_j^{(i)} - x_{j+1}^{(i)} \rangle = \sum_{e=xy \text{ edge of }T} \left \langle \sum_{i\in S_e} c(m_i)e_i,x-y \right\rangle
\]
Then, by applying inequality \eqref{eq:split} we have
\[
\left \langle \sum_{i\in S_e} c(m_i)e_i,x-y \right\rangle \leq c\left (\sum_{i\in S_e} m_i\right)\,\|x-y\|.
\]
Summing over all edges $e$ shows that the cost of the star is not greater than $C(T)$. Thus, the star network is a minimal Gilbert flow.
\end{proof}

Similarly for Examples~\ref{ex:dplus1} and~\ref{ex:dplus1p}, Theorem~\ref{th:stupid3} provides star with source of degree $2d$ with edges $\pm e_i$ for $p = 1/2$.

\section{Applications and further work} \label{sec:discussion}

\paragraph{Universal irrigation flow.} Having determined the maximal possible degrees of a source and a branching point one may try to construct a \textit{universal} solution to the Gilbert--Steiner problem, id est a solution which contains every possible combinatorial structure with the same $p$ and $d$ as a substructure. 
Such results for Steiner trees were obtained in~\cite{paolini2015example,cherkashin2023self,paolini2023steiner} and used in~\cite{basok2024uniqueness}.

At this point it seems to be possible only for a certain range of $p$, $d$ and only for irrigation setup. 
In particular, for $p = 1/2$ and $d = 2$ the following tree is universal.

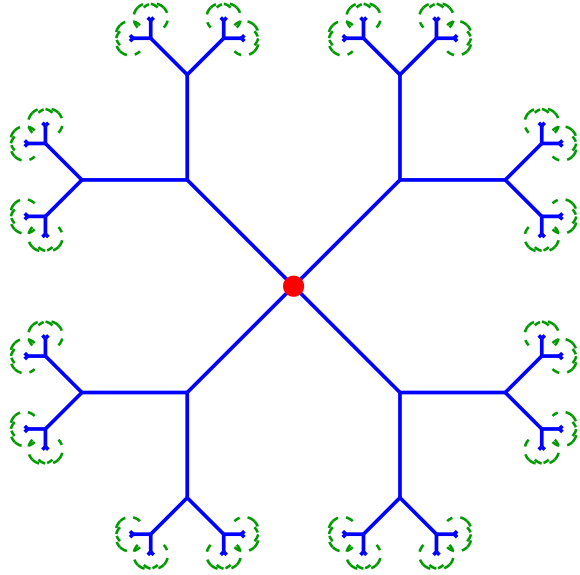
\begin{figure}[h]
    \centering  
    \input{pictures/UNi.tex}
    \caption{An example of a universal irrigation flow}
    \label{fig1}
\end{figure}

\paragraph{Classification of triple branching.}
It would be interesting to describe all cost functions for which the degree of \textit{planar} branching is exactly 3.
One of the motivations is the existence of \textit{exact} (but very slow) algorithm in this case.
There is no known algorithm in $\mathbb{R}^d$ for $d > 2$ (see Problem 15.12 in~\cite{bernot2008optimal}). One of the reasons is that the coordinates of the branching points, even for the Steiner problem for a general simplex in three-dimensional space, are no longer expressed in radicals~\cite{smith1992find}.

\paragraph{Picture in high dimensions.} 
Maximum possible degree of a source is heavily related with spherical codes and kissing number~\cite{BDM12}. The kissing number of a mathematical space is defined as the greatest number of non-overlapping unit spheres that can be arranged in that space such that they each touch a common unit sphere. The answer is known only in dimensions 1, 2, 3, 4, 8, 24.
Note that the presence of masses and conditions for $|I| > 2$ in Theorem~\ref{th:stupid3} make the situation even more intricate. 

A major question for the maximum possible degree of a branching point is whether there exists an upper bound which depends only on $d$. A positive resolution seems inaccessible to current methods.

\paragraph{Acknowledgments.} The research is supported by Bulgarian NSF grant KP-06-N72/6-2023.

\bibliography{main}
\bibliographystyle{plain}

\end{document}

%% file: pictures/mainCompet.tex
\begin{tikzpicture}[scale=3, every node/.style={font=\small}]

\coordinate (A) at (0,0);
\coordinate (B) at (2.8,0);
\coordinate (C) at (2,1);
\coordinate (P) at (1.6,0.45); 

\draw[thick] (A) -- (B) node[midway, below] {$m_1$};
\draw[thick] (A) -- (C) node[midway, above left] {$m_2$};

\draw[blue, thick] (A) -- (P) node[pos=0.5, below right] {$m_1 + m_2$};
\draw[blue, thick] (P) -- (B) node[midway, above right] {$m_1$};
\draw[blue, thick] (P) -- (C) node[midway, right] {$m_2$};

\foreach \pt in {A,B,C}
  \fill[red] (\pt) circle (0.5pt);

\fill[blue] (P) circle (0.5pt);

\end{tikzpicture}

%% file: pictures/UNi.tex
\begin{tikzpicture}[line width=1.4pt, scale=2, rotate = 45]
    \tikzset{branch/.style={blue}}
\def\lambda{0.7}
    \newcommand{\drawtree}[2]{%
        \ifnum#2>0
            \pgfmathsetmacro{\myscale}{\lambda^(#1 - #2)}
            \draw[branch] (0,0) -- (0,1);
            \begin{scope}[shift={(0,1)}]
                \begin{scope}[rotate=45,scale=\myscale]
                    \drawtree{#1}{\numexpr#2-1\relax}
                \end{scope}
                \begin{scope}[rotate=-45,scale=\myscale]
                    \drawtree{#1}{\numexpr#2-1\relax}
                \end{scope}
            \end{scope}
        \fi

        \ifnum#2=0
         \draw[green!60!black, dashed, line width=1pt, domain=90:180, samples=30]
            plot ({14*cos(\x) +14*sin(\x)},
                  {14*sin(\x) -14*cos(\x)});
        \fi
    }

    \foreach \angle in {0, 90, 180, 270} {
        \begin{scope}[rotate=\angle]
            \drawtree{6}{5}
        \end{scope}
    }

    \fill[red] (0,0) circle (2pt);

\end{tikzpicture}